\documentclass{amsart}
\usepackage{amssymb}
\usepackage{amsthm}
\usepackage{pb-diagram}
\usepackage{color}

\usepackage[textsize=tiny]{todonotes}

\newtheorem{theorem}{Theorem}[section]

\newtheorem{claim}[theorem]{Claim}
\newtheorem{fact}[theorem]{Fact}

\newtheorem{proposition}[theorem]{Proposition}
\newtheorem{lemma}[theorem]{Lemma}
\newtheorem{corollary}[theorem]{Corollary}

\newtheorem{question}[theorem]{Question}

\newtheorem*{question*}{Question}

\theoremstyle{definition}
\newtheorem{definition}[theorem]{Definition}

\newtheorem{question+}[theorem]{Question}

\newtheorem{defn}[theorem]{Definition}

 \newtheorem{example}[theorem]{Example}

\theoremstyle{remark}
\newtheorem{rmk}[theorem]{Remark}
\newtheorem{remark}[theorem]{Remark}

 \DeclareMathOperator{\intr}{int}

\newcommand{\WR}{\widetilde{\cal R}}

\newcommand{\la}{\langle}
\newcommand{\ra}{\rangle}

\newcommand{\sub}{\subseteq}

\newcommand{\fr}{\operatorname{fr}}

\newcommand{\Ima}{\ensuremath{\textup{Im}}}

\newcommand{\cal}[1]{\ensuremath{\mathcal{#1}}}
\newcommand{\Cal}[1]{\ensuremath{\mathcal{#1}}}

\newcommand{\res}{\ensuremath{\upharpoonright}}

\newcommand{\cl}[1]{\begin{overline}{#1}\end{overline}}

\newcommand{\es}{\ensuremath{\emptyset}}

\newcommand{\dom}{\ensuremath{\textup{dom}}}

\newcommand{\sm}{\setminus}

\newcommand{\Z}{\mathbb{Z}}
\newcommand{\N}{\mathbb{N}}
\newcommand{\Q}{\mathbb{Q}}
\newcommand{\R}{\mathbb{R}}

\title[Defining new linear functions]
{Defining new linear functions in tame expansions of the real ordered additive group}

\begin{document}

\author {Alex Savatovsky}

\address{Department of Mathematics and Statistics, University of Konstanz, Box 216, 78457 Konstanz, Germany}

\email{alex.savatovsky@uni-konstanz.de}


\date{\today}
\begin{abstract} We explore \emph{semibounded} expansions of arbitrary ordered groups; namely, expansions that do not define a field on the whole universe. 
For  $\cal R=\la R, <, +, \dots\ra$, a semibounded o-minimal structure and $P\sub R$ a set satisfying certain tameness conditions, we discuss under which conditions $\la\cal R,P\ra$ defines total linear functions that are not definable in \cal R. Examples of such structures that does define new total linear functions include the cases when $\cal R$ is a reduct of $\la\R,<,+,\cdot_{\upharpoonright (0,1)^2},(x\mapsto \lambda x)_{\lambda\in I\subseteq \R}\ra$, and $P= 2^\Z$,  or $P$ is an iteration sequence (for any $I$)  or $P=\Z$, for $I=\Q$. 

\end{abstract}
 \maketitle

\section{Introduction}
 This work, as the one of \cite{ES2} is at the intersection of two different direction in  model theory. The semibounded o-minimal structures that  were first  studied in the 90's by Marker, Peterzil, Pillay \cite{mpp, pet-reals, pss}  and others, they relate to 
Zilber's dichotomy principle on definable groups and fields,
and have continued to develop in recent years (see \cite{ed-sbd, el-sbd, pet-sbd}).  The second direction is that of \emph{tame expansions} $\WR$ of o-minimal structures \cal R. That are expansions of \Cal R which are not o-minimal, yet are still considered tame for some geometric/analytic, measure theoretic, model theoretic or descriptive set theoretic  properties shared by the sets definable in these structures. This area is much richer,  originating to A. Robinson \cite{rob} and van den Dries \cite{VDD1, vdd-dense}, it has largely expanded in the last two decades by many authors, and  includes broad categories of structures, such as d-minimal/ noiseless/ o-minimal open core  expansions of o-minimal structures, or more generally expansions of the real line  that does not define the whole projective hierarchy (see for example \cite{Mil1}, \cite{HM}, \cite{FM}, \cite{FHW}). 
\par  The work of this paper is the direct continuation of the one of  \cite{ES2} where we were studying a generalisation of semibounededness to tame although non-o-minimal settings for expansions of the real vector space $\la\R,<,+,(x\mapsto \lambda x)_{\lambda \in \R}\ra$. In this paper we generalize the result of \cite{ES2} to expansions of the real additive ordered group. Moreover, we give some necessary conditions for an expansion of the form $\la\cal R,P\ra$, for \cal R a semibounded expansion of the real additive ordered group and $P\subseteq \R$  not to define any total linear functions that are not definable in \cal R.
 As an application,  we obtain that for any reduct of  semibounded expansions $\WR$ of an o-minimal expansion of the real additive ordered group $\cal R$, such as $\WR=\la \cal R, 2^\Z\ra$   with $\cal R=\la  R,<,+, \cdot_{\res [0,1]^2},(x\mapsto \lambda x)_{\lambda\in I\subseteq \R}\ra$ or $\la\cal R,\Z\ra$  with $\cal R=\la  R,<,+, \cdot_{\res [0,1]^2}\ra$, every definable smooth (that is, infinitely differentiable) function  are already definable in \cal R. Especially, none of these structures defines total linear functions that are not definable in \Cal R.
\medskip
\par We now collect some definitions and state our results. We assume familiarity with the basics of o-minimality, as they can be found, for example, in \cite{VDD1}. A standard reference for semibounded o-minimal structures is \cite{ed-sbd}.




\textbf{For the rest of this paper, and unless stated otherwise, we fix an o-minimal expansion $\cal R=\la \R, <, +, \dots\ra$  of the real ordered group, and an expansion
$\WR=\la \cal R, \dots\ra$ of $\cal R$. By `definable'  we mean definable in $\WR$ , possibly with parameters. By $P$ we denote a subset of $\R$ of dimension $0$.}
\smallskip \par We denote by $\la\cal R,P\ra^\#$ the espansion of $\la\cal R,P\ra$ by predicates for any subsets of $P^k$ for every $k\in \N$. If $\WR=\la\Cal R,P\ra$, we denote by $\WR^\#=\la\cal R,P\ra^\#$. 
If \cal R is a real closed field, we call a set definable in \cal R \emph{semialgebraic}. We denote by $\R_{vec}=\la\R,<,+,(x\mapsto \lambda x )\ra$, $\R_{res}=\la\R,<,+,\cdot_{\res (0,1)^2}\ra$ and $\R_{sb}=\la\R_{vec},\R_{res}\ra$. We say that an expansion of $\R_{gp}$ is {\em purely linear} if it has the form  $\la\R,<,+,\big((x\mapsto\lambda x)\upharpoonright (0,1)\big)_{\lambda \in I},\big(x\mapsto\lambda x\big)_{\lambda\in J}\ra$ for some $J\subseteq I\subseteq \R$.
Given $\cal M$, an expansion of the real ordered additive group, we denote by $\Lambda(\cal M)$ the field of total linear functions definable in \cal M.
\medskip

\par Following \cite{ES2}, we define a generalization of semiboundedness to non o-minimal settings.

\begin{defn}\label{def sb}
Let $\cal M = \la M, <, +, \dots\ra$ be an expansion of an ordered group. We call $\cal M$ \emph{semibounded} if there is no definable field whose domain is the whole $M$.
\end{defn}

\begin{rmk}In \cite{ES2} we asked the question wether or not this definition was equivalent to defining no  pole (that is a bijection between a bounded interval and $\R$). In the o-minimal setting the two properties are equivalent. However in our setting they  are not and we exhibit a counterexample in Section \ref{discussion}. 
\end{rmk} 
 For $\cal M$ a semibounded structure, we denote by $\Lambda(\cal M)$ the field of total linear functions definable in \cal M (possibly with parameters).
For any set $X\sub R^n$, we define its \emph{dimension} as the maximum $k$ such that some projection of $X$ to $k$ coordinates has non-empty interior.
\medskip
\par In \cite{ES2}, we introduced certain  tameness properties for expansions of the real additive ordered group ((DP) and (DIM) below). 
\begin{defn}
Let $Y\sub X\sub R^n$ be two sets. We say that $Y$ is a \emph{chunk of $X$} if it is a  cell definable in \Cal R, $\dim Y= \dim X$, and for every $y\in Y$, there is an open box $B\sub R^n$ containing $y$ such that $B\cap X\sub Y$. Equivalently, $Y$ is a relatively open cell in $X$  with $\dim Y=\dim X$ and definable in \Cal R.
\end{defn}


\begin{defn} We say that $\WR$ has the \emph{decomposition property} \textup{(DP)} if for every definable set $X\sub \R^n$,
\begin{enumerate}
  \item[(I)] there is a family $\{Y_t\}_{t\in R^m}$ of subsets of $R^n$ that is  definable in \cal R, and a definable set $S\sub R^m$ with $\dim S=0$, such that $X=\bigcup_{t\in S} Y_t$,

  \item[(II)] $X$ contains a chunk.
\end{enumerate}
\end{defn}

\begin{defn}
We say that $\WR$ has the \emph{dimension property} (DIM) if for every  family $\{X_t\}_{t\in A}$ definable in \Cal R, and definable set $S\subseteq A$ so that  $\dim S=0$, we have
$$\dim \bigcup_{t\in S} X_t=\max_{t\in S} \dim X_t.$$\end{defn}
\begin{defn}
We say that $\WR$ is {\em noiseless} if, given a definable set $X$, then it has interior or it is nowhere dense. Equivalently if for $X\subseteq \R$, $X$ has interior or is nowhere dense (see \cite{Mil1}). 
\end{defn}
 \begin{remark}
 Note that Fornasiero calls noiseless structures i-minimal. It is also the denomination we have used in \cite{ES2}.
 \end{remark}
As we saw in \cite[Lemma 4.4, remark 4.5]{ES2}, (DIM) is equivalent to noiselessness. Morevover, in our setting noiseless implies (DPI).
\begin{fact}\label{DPI}(see \cite[main theorem]{S})Let $\WR=\la\cal R,P\ra$ being noiseless. Then $\WR^\#$ has (DPI). Moreover, given a definable set $X$, we can find a family  of cells $\{X_t:\,t\in A\}$ definable in \cal R and $S\subseteq A$, a set of dimension $0$ so that $X=\bigcup_{t\in S}X_t$ and for every projection $\pi$, for every  $t,t'\in S$ either $\pi(X_t)\cap \pi(X_{t'})=\es$ or they are equal.
 \end{fact}
Moreover, in our setting, we may further simplify the  assumptions.
\begin{defn}
 We say that $\WR$ has the {interior or isolated point property} if given any definable set $X\subseteq \R$, then $X$ has interior or has an isolated point.    
\end{defn}

\begin{proposition}
Let $\WR=\la\Cal R,P\ra$ has the interior or isolated point property. Then $\WR^\#$ is noiseless and has (DP). 
\end{proposition}
\begin{proof}
noiselessness is straightforward since a dense set without interior has no isolated point. (DPI) follows from Fact \ref{DPI}. 
\par For (DPII), let $X\subseteq \R^n$ be a set of dimension $m$. We do an induction on $n$. For $n=0$, the result is trivial. We assume that $n>1$ and that for any $m'\leq n'<n$, (DPII) holds. The case $m=n$ is straightforward since $X$ would have interior and that an open box is a chunk of $X$. We assume $m<n$. By a simple induction, we may further assume that $m=n-1$. Let $\pi$
 be the projection on the first $n-1$-coordinates and without loss of generality, we may assume that $\dim(\pi(X))=n$.
 By Fact \ref{DPI}, there is a family of cells $\{X_t:\, t\in A \}$  and $S\subseteq A$ a small set that satisfy the conditions of Fact \ref{DPI}. Let $t\in S$ so that $\dim(\pi(X_t))=n$ and let $Z=\pi(X_t)$. By assumption on $\{X_t:\; t\in S\}$, $\pi^{-1}(Z)\cap X$ is composed of disjoint graphs of functions with domain $Z$. Ie $\pi^{-1}(Z)\cap X=\bigcup_{t\in S'}X_{t}$ for some $S'\subseteq S$, for every $t\in S'$ $X_t$ is the graph of a continuous function with domain $Z$ and for every $t,t'\in S'$, 
 $X_t\cap X_{t'}=\es$. Since for every $x\in Z$ there is at least one isolated point in $\pi^{-1}(x)\cap X$, the set $Y=\{x\in \bigcup_{t\in S'}X_t:\; \text{$x$ is isolated in $\pi^{-1}(\pi(x))\cap X$}\}$ 
 has dimension $n$. By Fact \ref{DPI}, we may find a compact set $W\subseteq Y$, definable in \cal R and so that $\dim(W)=m$. \par We just have to show that $W\cap \pi^{-1}(\intr(\pi(W)))$ is a chunk of $X$. Since $W$ is compact, $\{\d(x,\pi^{-1}(\pi(x))\cap X):\, x\in W\}$ has an infimum that is not $0$. This  gives us the result.   
 \end{proof}
\begin{rmk}
Note that  (DPII) implies the interior or isolated point property. Therefore, in the case where \Cal R expands $\R_{gp}$, the conditions of \cite{es}, \cite{ES2} are equivalent to (DPII) alone.
\end{rmk}
Therefore we may rephrase the conditions of \cite{ES2} into $\WR$ has the interior or isolated point property. 
\bigskip
\par In Section \ref{proof main}, we generalize \cite[Theoreme 1.5]{es} to our setting.
\begin{theorem}\label{main1}
Let $\WR\subseteq \la\Cal R,P\ra^\#$ having the interior or isolated point property. Then every definable $\cal C^\infty$-function with a domain definable in \cal R is definable in $\la\Cal R,\Lambda(\WR)\ra$. 
\end{theorem}

 Moreover, if \Cal R is purely linear, more can be said:
\begin{theorem} \label{main2} Let $\cal R=\la\R,<,+,\big( (x\mapsto \lambda x)_{\upharpoonright(0,1)}\big)_{\lambda \in I},
(x\mapsto \lambda x)_{\lambda \in J}\ra$ be purely linear.
Let $\WR\subseteq \la\cal R,P\ra^\#$ having the interior or isolated point property. Then every definable $\cal C^1$-function  with a domain definable in \cal R is definable in $\la\Cal R,\Lambda(\WR)\ra$. Moreover, $\Lambda (\WR)\subseteq I$. 
\end{theorem}
\smallskip
\par With these results, we may further study some examples.
\begin{definition}[{\cite{MT}}] Let $f:\R\to \R$ be a  bijection definable in \cal R, and $f^n$ the $n$-th compositional iterate of $f$. We say that \cal R is \emph{$f$-bounded} if for every  function definable in \Cal R $g : \R\to \R$, there is $n\in \N$ such that ultimately $f^n>g$.

Let $c\in\R$ and $f$ a  function definable in \cal R such that \cal R is $f$-bounded, and such that $(f^n(c))_n$ is growing and unbounded. We call $(f^n(c))_n$ an {\em iteration sequence}.
 \end{definition}
Before studying examples, we need to define two important class of structures.
\begin{definition}
We say that $\WR$ is d-minimal if given any definable family $\{X_t:\; t\in A\}$ of subsets of $\R$ of dimension $0$, there is $N$ so that for every $t\in A$, $X_t$ is the union of at most $N$ discrete sets. Note that if $\WR$ is d-minimal then it has the interior or isolated point property. Note also that, as shown in \cite{FM}, if $\WR$ is d-minimal then $\WR^\#$ is d-minimal too.
\par We say that $\WR$ is locally o-minimal if  any bounded definable set of dimension $0$ is finite. Note that local o-minimality implies d-minimality since a set of dimension $0$ needs to be discrete.
\end{definition} 
\medskip 
\begin{example} \label{ex}We consider the following structures:
\begin{enumerate}
 \item $\la \cal R,\Z\ra^\#$ with $\Lambda(\cal R)=\Q$,
  \item   $\la \cal R,P\ra^\#$, where $P$ is an iteration sequence for some o-minimal expansion $\cal R'$ of $\cal R$.
  \item $\la\cal R, \alpha^\Z\ra^\#$, for $\alpha>0$ 
  \item $\la\R_{vec},P \ra^\#$, where $P>0$ is the range of a decreasing sequence with limit $0$.
  
\end{enumerate}
\end{example}

\begin{rmk}\label{rmk alpha^N}
 An iteration sequence for \cal R has the form $(f^n(c))_n$ for $f:\R\to\R$ a  growing function definable in 
 \cal R such that $\cal R$ is $f$-bounded. Since \Cal R is semibounded, it is ultimately affine and thus we may assume that $f$  is of the form $x\mapsto ax+b$ (for $a>1\in \Lambda(\cal R)$, $b\in \R$). Therefore, for $c\in \R$, $f^n(c)=a^nc+\sum_{0\leq i<n}a^ib$. Then $$f^n(c)=a^nc+\frac{a^n-1}{a-1}b=\frac{1}{a-1}\big(a^n((a-1)c+b)-b\big).$$
 Thus $$\la\cal R, (x\mapsto \beta x),(f^n(c))_n\ra= \la \cal R, (x\mapsto \beta x),(a^n)_n\ra$$ for $\beta=(a-1)c+b$.
 \par Therefore, in Example \ref{ex}, the special case of ($2$)  where $P$ is an iteration sequence for $\Cal R$ is actually a reduct of ($3$). The other cases of ($2$) includes $P=\alpha^\N$ where $\alpha \notin \Lambda(\cal R)$ and any iteration sequences for $\cl{\R}$ or other non semibounded o-minimal structures. 
\end{rmk}
\medskip
\par In section \ref{d-min}, we prove the following proposition:
\begin{proposition}\label{ex d-min}The structures of  Example \ref{ex} are d-minimal. moreover, ($1$)-($2$) are locally o-minimal.
 \end{proposition}
\smallskip
\par In Section \ref{sec ex}, we prove the following Proposition:
\begin{proposition}\label{prop-examples}
Let $\WR$ be ($1$)-($3$) of  Example \ref{ex}. Then a definable smooth function over a semialgebraic domain is definable in \Cal R.
\end{proposition}
 Moreover, we discuss some conditions that would allow $\la\Cal R,P\ra^\#$  to define linear functions that are not definable in $\cal R$.
\par We restrict further the shape of \Cal R and establish the following proposition:
\begin{proposition}\label{prop-ex-lin}
Let $\cal R$ be purely linear. Let $\WR$ be any of the structures of Example \ref{ex}. Then every definable $\Cal C^1$ functions with a  domain definable in \Cal R is definable in \Cal R. 
\end{proposition}
\smallskip
\par In Section \ref{discussion}, we discuss some questions and give some examples of  semibounded d-minimal structure that are not locally o-minimal and so that for every algebraic set $X\subseteq \R^n$ not definable in \Cal R,  $\la\WR,X\ra$ defines any bounded set in the projective hierarchy. We also discuss the difference between being semibounded in the sens of Definition \ref{def sb} and defining no poles. We finish with some matural questions that arise from this work.
\smallskip

$ $\\
\noindent\emph{Structure of the paper.}
In Section \ref{sec-prel}, we fix some notations, and establish  basic properties for semibounded structures.
 In Section \ref{proof main}, we prove Theorems \ref{main1} and \ref{main2}.  In Section \ref{d-min}, we prove Proposition \ref{ex d-min}.  In Section \ref{sec ex}, we prove Propositions \ref{prop-examples},  \ref{prop-ex-lin}. In Section \ref{discussion}, we discuss some questions.

\medskip
\par \textbf{aknowledgement:} We would like to thank Philipp Hieronymi and Eric Walsberg for raising the questions that led to this work. We would also like to specially thank Pantelis Eleftheriou for asking questions, for his comments and his corrections.    

\section{Preliminaries}\label{sec-prel}

In this section, we fix some notation and prove basic facts about semibounded structures.

 If $A, B\sub \R$, we denote $\frac{A}{B}=\{a/b : a\in A, b\in B\}$. If $t\in \R$, we write $\frac{A}{t}$ for $\frac{A}{\{t\}}$.
By a $k$-cell, we mean a  cell of dimension $k$.  If $S\sub \R^n$ is a set, its closure is denoted by $\cl S$, with sole exception  $\overline \R$, which denotes the real field. By an open box $B\subseteq \R^n$, we mean a set of the form
$$B=(a_1, b_1)\times \ldots \times (a_n, b_n),$$
for some $a_i< b_i\in \R\cup\{\pm\infty\}$. By an open set we always mean a non-empty open set.  

When we write equality between two structures, we mean the two structures have the same collection of definable sets.

Let $\cal M $ be an expansion of $\R_{gp}$. By $P_{ind}(\cal M)$, we mean the structure induced on $P$ in $\la\cal M,P\ra$ that is $$\la P,\{S\subseteq P^k:\,\text{$S$ is definable in $\la\cal M,P\ra$}\}.$$


\begin{definition}
\par We say that $X\subseteq \R^n$ is a \emph{cone} if there are a bounded definable set $B\subseteq \R^n$, $v_1,\ldots,v_l\in \Lambda(\cal R)^n$ some linearly independent vectors over $R$ such that
$$X=B+\sum_{a\in (\R^{>0})^l}\sum_i a_iv_i:=\{x\in \R^n\,:\; \exists b\in B,\, a\in (\R^{>0})^l,\; x=b+\sum_ia_iv_i\}$$
Moreover,  for every $x\in X$, there are unique $b\in B$ and $a \in (\R^{>0})^l$ with $$x=b+\sum_ia_iv_i.$$
\end{definition}

\begin{fact}[Structure Theorem: decomposition into cones {\cite[Fact 1.6 (5)]{ed-sbd}}]\label{ed-sbd}
 Let $X\subseteq \R^n$ be a set definable in \Cal R. Then $X$ can be decomposed into finitely many cones.
\end{fact}

\begin{fact}\label{str function}
Let $f:X\subseteq \R^n\rightarrow \R$ be a function definable in \Cal R. Then there is an interval $B\subseteq \R$ and an affine function $\lambda: \R^n\rightarrow \R$, $x\mapsto \sum_i\lambda_ix_i+b$ where $\lambda_i\in \Lambda(\cal R)$ such that for every $x\in X$, $f(x)\in \lambda(x)+B$.
\end{fact}
\begin{proof}
Easy to see, using \cite[Fact 1.6]{ed-sbd}.
\end{proof}
\begin{proposition}\label{red sb}
Let $\cal R$ be a reduct of $\R_{sb}$. Then \cal R has the form $\la \cal R',(x\mapsto \lambda x)_{\lambda\in I}\ra$, for some semibounded o-minimal expansion of $\R_{gp}$ so that  $\Lambda(\Cal R')=\Q$. 
\end{proposition}
\begin{proof}
First, as a reduct of $\R_{sb}$, $\Cal R$ is o-minimal and semibounded.  Let $X$ be a set definable in \Cal R. By a simple induction on the structure of the cell, we may assume that $X$ is the graph of a function $f$. By Fact \ref{ed-sbd} we may assume that there is $C$ a bounded set definable in $\cal R$ and $v_1,\ldots,v_n\in \R^{n+m}$ some linearly independent vectors so that $$\Gamma(f)=C+\sum_{a\in (\R^{>0})^n}\sum_ia_iv_i.$$
It is clear that $f$ is definable in $\la\Cal R,(x\mapsto v_{i,j}x)_{i\leq n,j\leq n+m}\ra$. 
What remains to show is that $\WR$ defines $(x\mapsto v_{i,j}x)$ for every $i,j$.
 By permuting coordinates, we just have to show that $x\mapsto v_{1,1}x$ is definable in $\WR$. Moreover, without loss of generality, we may assume that $$\Gamma(f)=\sum_{a\in (\R^{>0})^n}\sum_ia_iv_i.$$
\par We proceed by induction on $n$. For $n=1$, we may assume $f:\R^{>0}\to \R^m$, $x\mapsto (xv_{1,j})_j$ and $$ v_{1,1}x=\pi_1(f(x)).$$
\par We assume the result to holds for $n$ and prove it for $n+1$. Let $Y=\sum_{a\in (\R^{>0})^{n}}\sum a_iv_i$. Observe that $Y\subseteq \fr(\Gamma(f))$, $Y$ is  definable in $\WR$ and $\dim(Y)=n$. We apply the induction hypothesis to $Y$ to get the result. 
\end{proof}
We finish this section by exposing a central result by Friedmann and Miller:

\begin{definition}[\cite{FM}]
We say that a set $Q\subseteq \R$ is \emph{sparse} if for every  function definable in \cal R $f: \R^k\rightarrow \R$,  $\dim f(Q^k)=0$.
\end{definition}
\begin{rmk}
 If $\la\cal R,P\ra$ is i-minimal, then $P$ is sparse.
\end{rmk}

\begin{fact}[{\cite[Last claim in the proof of Theorem A]{FM}}]\label{FM} Assume $P\sub \R$ is sparse.
Let $A\subseteq \R^{n+1}$ be definable in $\WR^\#$  such that for every $x\in \R^n$, $A_x$ has no interior. Then there is a function definable in \cal R $f:\R^{m+n}\rightarrow \R$ such that for every $x\in \R^n$, $$A_x\subseteq \cl{f(P^m\times\{x\})}.$$
\end{fact}

\section{Proof of Theorem \ref{main1}, \ref{main2}}\label{proof main}
\par The goal of this section is only to establish Theorems \ref{main1} and \ref{main2}.

\begin{theorem} We assume that $\WR\subseteq \la\Cal R,P\ra^\#$ has the interior or isolated point property.
Let $f:X\subseteq \R^n\to \R^m$ be a $\cal C^\infty$-function definable in $\WR$ with a domain definable in \cal R. Then $f$ is definable in $\la \cal R,\Lambda(\WR)\ra$. 
\end{theorem}
\begin{proof}
First, by \cite{ES2}[Theorem 1.7], $f$ is definable in $\R_{sb}$ (by the exact same proof, actually, since it is not necessary that $\la\R_{sb},P\ra$ has the interior or isolated point property. Take for example ($1$) of Example \ref{ex}). By Proposition \ref{red sb}, $\la\cal R,f\ra=\la\Cal R,(x\mapsto \lambda x)_{\lambda\in I}\ra$ for some $I$ and by definition of $\Lambda(\WR)$ and since $f$ is definable in $\WR$, $f$ is definable in $\la\Cal R,\Lambda(\WR)\ra$. 
\end{proof}

\begin{proposition}\label{Z lin}
Let $\cal R$ be a linear reduct of $\R_{sb}$ of the form $\la\R,<,+,\big( (x\mapsto \lambda x)_{\res (0,1)}\big)_{\lambda\in I},(x\mapsto \lambda x )_{\lambda \in J}\ra$. We assume that $\la \cal R,P\ra$ has the interior or isolated point property and $\WR\subseteq \la\cal R,P\ra^\#$. Let $f$ be a $\Cal C^1$ function that has a domain that is definable in \cal R and that is definable in $\WR$. Then $f$ is definable in $\la\Cal R,\Lambda(\WR)\ra$. Moreover, $\Lambda(\WR)\subseteq I$.
\end{proposition}
Before proceeding to the proof, we first need a basic fact. We don't prove it.
\begin{fact}\label{lin dim 1}
Let $f: \R^n\to \R$ be a function. Then $f$ is linear if and only if for every $x\in \R^{n-1}$, $\sigma$ a coordinate permutation, $g:y\mapsto f(\sigma(x,y))$ is linear.
\end{fact}
\begin{proof}(of Proposition \ref{Z lin})
By Fact \ref{lin dim 1}, we may assume that $n=1$ and that $\dom(f)=\R$.
 Moreover, since $\{X_t:\; t\in S\}$ is a family of linear functions, there are  $\{a_t, b_t\in \R:\;t\in S\}$ so that for every $t\in S$ $X_t$ is the restriction  of $x\mapsto a_tx+b_t$ to its domain.
 Let $S'=\{t\in S:\; \dim(X_t)=1\}$.  By (DIM), $$\dim\big(\Gamma(f)\sm\bigcup_{t\in S'}X_t\big)=0$$
 and since $f$ is continuous, $$\Gamma(f)=\cl{\bigcup_{t\in S'}X_t}.$$ Since $f$ is $\Cal C^1$, $$g:\bigcup_{t\in S'}\dom(X_t)\to \R, \,x\in \dom(X_t) \mapsto a_t $$
 can be extended by continuity to $\R$.  Therefore, since $\dim(\Ima(f'))=0$, $\Ima(f')=\{*\}$  and for every $t,t'\in S'$ $a_t=a_{t'}$. Moreover by continuity of $f$, $b_t=b_{t'}$ and  $f$ is then linear with coefficients in $I$ and is definable in $\la \R,<,+,(x\mapsto \lambda x)_{\lambda\in I}\ra$.
\end{proof}

\begin{rmk}
Note that, assuming definable completeness (see \cite{Mil-IVT}. Equivalently, asking for the intermediate value property for definable continuous functions, or that any interval is definably connected), Theorems \ref{main1}, \ref{main2} holds for non archimedian expansions of a semibounded expansion of an additive abelian ordered group by a set of dimension $0$ so that the interior or isolated point property holds.
\end{rmk}

\section{d-minimality}\label{d-min}
In this section, we establish d-minimality for the structures of Example \ref{ex}. 

\subsection{local o-minimality of ($1$)-($2$)}In this subsection,. we establish local o-minimality for 
 $ \la\cal R,\Z\ra$ (if $\Lambda(\Cal R)=\Q$) and $\la\Cal R,P\ra$ with $P$ an iteration sequence for some expansion of \cal R. 


\begin{fact}\label{R0Z dmin}
The structure $\la \R, <, +, \Z\ra$ is locally o-minimal.
\end{fact}
\begin{proof}
This is a consequence of the local o-minimality of $\la\cal R, <, +, \big( x\mapsto \sin(2\pi x)\big)\ra$, proved in \cite[Theorem 2.7]{ToVo},  since $\la\cal R, <, +,\Z\ra$ is a reduct of it.
\end{proof}


\begin{lemma}\label{R1Z dmin}
Let $P\sub \R$ with $\dim P=0$ and assume that $\la \R,<,+,\Lambda(\cal R),P\ra$ is locally o-minimal. Then $\la \cal R,P\ra$ is locally o-minimal.
\end{lemma}
\begin{proof}Let $X\subseteq \R$ be a definable set of dimension $0$. We show that $X$ is discrete.
Since $P$ is countable, it is sparse and by Fact \ref{FM} there is a function definable in \Cal R $f:\R^k\to \R$ such that $X\subseteq \cl{f(P^k)}$.
  \par Let $B$ be a bounded interval. We just have to  show that $B\cap f(P^k)$ is finite.
  By Fact \ref{str function}, there is $g:\R^k\to \R$, a linear function (definable in \Cal R) and a bounded interval $C=(-a,a)$ so that for every $x\in P^k$, $f(x)\in g(x)+B$. Since $\la \R,<,+,\Lambda(\cal R),P\ra$ is locally o-minimal, $g(P^k)\cap (B+C)$ is finite, $f(P^k)\cap B$ is finite too and we have the result.

\end{proof}
\begin{rmk} Note that assuming that $\la\R,<,+,\Lambda(\cal R),P\ra$ is locally o-minimal and not only $\la\R,<,+,P\ra$ is necessary; for example, take  $\la\R,<,+,(x\mapsto \sqrt{2}x)\ra$, that is a semibounded o-minimal expansion of $\la\R,<,+\ra$ but $\la\R,<,+,(x\mapsto \sqrt{2}x),\Z\ra$ is not d-minimal as shown in \cite{Hiero two gps}.
\par We could ask if the same result holds for d-minimality instead of local o-minimality. It is not the case, see Section \ref{discussion} for a counterexample.
\end{rmk}
\medskip
\par In order to prove local o-minimality for Example \ref{ex}($2$) , we first need some more general proposition about the structure induced on $P$ by $\WR$.
\begin{proposition}\label{Pind}
The structure induced on $P$ by $\la\cal R,P\ra$ is $\la P,<\ra$.
\end{proposition}
\begin{proof}First, note that if $P$ is an iteration sequence for $\Cal R'$ some expansion of \Cal R, and that the structure induced on $P$ by $\Cal R'$ is $\la P,<\ra$ then the structure induced on $P$ by $\Cal R$ is also $\la P,<\ra$. Therefore  we may assume that $P$ is an iteration sequence for $\cal R$. 
\par In \cite[Proof of Lemma 1.1]{For dcnc}, using the quantifier elimination result from \cite{MT}, Fornasiero is actually showing that $P_{ind}$ is o-minimal. By \cite{PS disc o-min 1}, $P_{ind}$ is strongly o-minimal and by \cite[Theorem 3.1]{PS disc o-min 2}, $P_{ind}(\WR)=\la P,<\ra$.
\end{proof}
In the following, $s:P\to P$ denote the successor function. 
\begin{definition}
Let $X\subseteq P$ be definable in $\la P,<\ra$. We say that $X$ is an {\em s-cell} if it is a point or an iterval of the form $(a,+\infty)$ for some $a\in \P$. Let $X\subseteq P^n$ and let $\pi$ be the projection on the first $n-1$-coordinates. We say that $X$ is an $s$-cell if, $\pi(X)$ is one and $X$ either has the form $$\{(x,y)\,:\; x\in \pi(X),\, y=f(x)\}$$
or $$\{(x,y)\,:\; x\in \pi(X),\, g(x)>y>f(x)\}$$
for some function $f,g:P^{n-1}\to P\cup\{\infty\}$ of the form $x\mapsto a$ for some $a\in P\cup\{\infty\}$ or $x\mapsto s^k (\pi_i(x))$ for some $i\in \{1,\ldots,n-1\}$ and $k\in \Z$ and where $\{\#\{P\cap (f(x),g(x))\}:\; x\in \pi(X)\}$ is unbounded.
\end{definition}
\begin{fact}\label{Fact alpha}
Let $X$ be a set definable in $\la P,<\ra$. By a simple quantifier elimination result we get that there is decomposition of $X$ into finitely many s-cells.
\end{fact}
\begin{lemma}\label{s-c dim 1}
Let $X\subseteq P^{n}$ be an infinite s-cell. Then there is $Y\subseteq X$, $K\subseteq \{1,\ldots n\}$, $a_i\in P$ and $j_i\in \Z$ for $i\in \{1,\ldots,n\}$, $a\in P$  so that
$$Y=\{(x_1,\ldots,x_n):\; \text{there is $x>a\in P$ for $i\in K$ $x_i=s^{j_i}x$, and for $i\notin K$, $x_i=a_i$}\}.$$
\end{lemma}
\begin{proof}
We do an induction on $n$. For $n=1$, since $X$ is infinite the result follows from the definition of an s-cell. We assume the result holds for $n$ and that $X\subseteq P^{n+1}$. For $\pi$ the projection on the first $n$-coordinates, if $\pi(X)$ is finite then it is a singleton $\{a\}$ and since $X$ is infinite, $X=\{(a,y):\; y>b\}$ for some $b\in P$ and we have the result.

 \par If $\pi(X)$ is infinite then we apply the induction hypothesis to $\pi(X)$ to get $X_1$ of the wanted form. It is easy to see that there are some functions  $f,g$ definable in $\la P,<\ra$ of the form $x\mapsto s^k (\pi_i(x))$ or constant so that $\pi^{-1}(X_1)\cap X$ either has the form  $$\{(x,y):\;  \text{there is $z>a\in P$ for $i\in K$ $x_i=s^{j_i}(z)$, for $i\notin K$, $x_i=a_i$ and $y=f(z)$}\}$$
 that is of the right form, or has the form $$\{(x,y):\;  \text{there is $z>a\in P$ for $i\in K$ $x_i=s^{j_i}(z)$, for $i\notin K$, $x_i=a_i$ and $g(z)>y>f(z)$}\}.$$
 Then $$\{(x,y):\;  \text{there is $z>a\in P$ for $i\in K$ $x_i=s^{j_i}(z)$, and for $i\notin K$, $x_i=a_i$ and $y=s( f(z))$}\}\subseteq X$$ and has the right form.
\end{proof}
\begin{corollary}\label{loc o-min}The structures ($1$)-($2$) of Example \ref{ex} are locally o-minimal.

\end{corollary}
\begin{proof}
By Lemma \ref{R1Z dmin}, we just have to show that $\la\cal R,P\ra$ are locally o-minimal for $\cal R$ a pure linear structure. For ($1$), it is by
Fact \ref{R0Z dmin}.
For ($2$), towards a contradiction, we assume that  there is a definable set $X\subset \R$ of dimension $0$ so that $0$ is an accumulation point of $X$. By Fact \ref{FM}, there is a linear function $f:\R^k\to \R$ so that $X\subseteq \cl{f(P^k)}$.
\begin{claim}\label{cl1}
There is a decreasing sequence $(x_n)_n\in X^\N$ so that $\lim_n(x_n)=0$ and $\{x_n:\; n\in \N\}$ is definable.
\end{claim}
\begin{proof}By \cite{MT}, $\la\Cal R',P\ra$ is d-minimal and, as a reduct, $ \la\cal R,P\ra $ is d-minimal too. Therefore,
we can do an induction on the number $n$ of discrete sets needed to decompose $Y=\{x\in \cl{X}:\; x>0\}$. Without loss of generality, we may  assume that: $$ \text{ (*) the number of discrete sets needed to decompose   $B\cap Y$ for any boxes $B\ni 0$ is $n$}.$$  If  $n=1$, since $\cl{X}$ is closed, then $Y$ is the range of a sequence that goes to $0$.
\par We assume that the result holds for $n$ and prove it for $n+1$. Let $Y'=\{x\in Y:\; \text{$x$ is not isolated in $Y$}\}$. By a simple induction, one can show that $Y'$ admits a decomposition into $n$ discrete sets. By definition $Y'$ is closed in its convex hull  and by (*), $0\in \cl{Y'}$. We apply the induction hypothesis to $Y'$ to get the result.
 \end{proof}

 Let $Y=\{x_n:\; n\in \N\}$ be the set produced by Claim \ref{cl1}. Note that d-minimal structures have definable choice.
For $n\in \N$ let $$\text{$y_n$ be a choice among $f(P^k)\cap [x_{n+1},x_n)$ if it exists and $y_n=y_{n-1}$ otherwise}\}.$$
By replacing $Y$ by $Y'=\{y_n:\; n\in \N\}$, we may assume that $Y\subseteq f(P^k)$. Let $\rho:Y\to P^k$ be a choice function so that $\rho(y)\in f^{-1}(y)$. Let $Z=\rho(Y)$. Since $Y$ is infinite, $Z\subseteq P^k$ is too and by Fact \ref{Fact alpha}, it contains an infinite s-cell. By Lemma \ref{s-c dim 1}, there is $Z'\subseteq Z$ so that $Z'$ has the form  $$\{(x_1,\ldots,x_n):\; \text{there is $x>a\in P$ for $i\in I$ $x_i=s^{j_i}(x)$, and for $i\notin I$, $x_i=a_i$}\}.$$
Thus $f(Z')=\mu(P^{>a})$ for some limear function $\mu$ and $f(Z')$ is closed and discrete. Moreover, since $Z'$ is infinite, that $f(Z')\subseteq f(Z)=Y$ and that $f_{\res Z'}$ is an embedding, $f(Z') $ is the range of a decreasing sequence with limit $0$. That is a contradiction.

 \end{proof}
\subsection{d-minimality of any expansion of  a linear structure by any sequence with limit $0$}

\begin{theorem}\label{T1}
We assume that $P\subseteq \R^{>0}$ is discrete and closed in its convex hull (thus, it has at most two accumulation points, and we assume one to be $0$ and $P>0$). Moreover, we assume that $\la\cal R,P^{>1}\ra$ is locally o-minimal. Then $\WR^\#$ is d-minimal.
\end{theorem}
\begin{proof}
By \cite{FM}, it is sufficient to show that $\WR$ is d-minimal.
Since $P$ is countable, it is sparse and by Fact \ref{FM}, we just need to show that, for $f:\R^{n+m}\to \R$ definable in \cal R, there are $g_i:\R^{n_i+n}\to \R$ some  functions definable in \Cal R so that for every $a\in \R^n$, $\cl{f(a,P^m)}$ is the union of  finitely many discrete sets. In the case where $P\cap [0,1]$ is finite then $\WR$ is locally o-minimal and the result follows. Thus, we may assume that $P\cap [0,1]$ is the range of a decreasing sequence $(x_n)$ with limit $0$.
\par Since $f$ is definable in \cal R, it has the form $x\mapsto \sum_i\lambda_i x_i+b$. Since for $a$, $a'\in\R^n $, $f(a,P^k)$ and $f(a',P^k)$ are translate of each others we may assume that $n=b=0$.  
\par We proceed by induction on $m$. For $m=0$, the result is trivial. We assume that the result holds for $m<k$ and prove it for $k$. We first consider the image under $f$ of $Z=(P^m)\cap [0,1]^m $. Let $x$ be a non-isolated  point of $\cl{f(Z)}$ and $\big((x_{i,s})_{i\leq m}\big)_s$ be a sequence so that $$\lim_sf((x_{i,s})_i)=x$$ and that there is no subsequence that is ultimately constant in each coordinates. For $A\subseteq\{1,\ldots, k\}$, let $f_A:\R^{|A|}\to \R$, $x\mapsto \sum_{i}a_{j_i}x_i$, where $j_i$ is the $i$-th element of $A$. By induction, $$\bigcup_{A\subsetneq \{1,\ldots,k\}}\cl{f_A(P^{|A|})}\text{ is the union of finitely many discrete sets.}$$ We show that $x\in \bigcup_{A\subsetneq \{1,\ldots,k\}}\cl{f_A(P^{|A|})}$.
\par  By taking subsequences,  we may assume that for every $i$, $(x_{i,s})_i$ is monotone (constant or decreasing). By assumption they can not all be ultimately constant. Thus, there is $i$ so that $(x_{i,s})$ is strictly  decreasing: $( \lambda_ix_{i,s})_s$ goes to zero and $x\in \bigcup_{A\subsetneq \{1,\ldots,k\}}\cl{f_A(P^{|A|})}$. Therefore, $$\cl{f(Z)}\setminus \bigcup_{A\subsetneq \{1,\ldots,k\}}\cl{f_A(P^{|A|})} \text{ is discrete}$$
and we have d-minimality by induction.
\par Now, we consider the case  $\cl{f\big((P^{>1})^l\times(P^{<1})^n\big)}$ for some $l\neq 0$. Let $K=\{1,\ldots,l\}$, $K'=\{l+1,\ldots,k\}$. Since $f$ is linear, there is a box $B$ of finite radius so that $$f_{K'}((P^{<1})^n)\subseteq B.$$ Moreover, by local o-minimality of $\la\cal R,P^{>1}\ra$, for every $x\in \cl{f\big((P^{>1})^l\times(P^{<1})^n\big))}$, there are only finitely many $y\in (P^{>1})^l$ so that $x\in f_K(y)+B$ and thus $$x\in \cl{f_{K'}((P^{<1})^n)}+f_K(y)$$ 
for some $y\in (P^{>1})^l$. Therefore $$\cl{f\big((P^{>1})^l\times(P^{<1})^n\big))}=\cl{f_{K'}((P^{<1})^n)}+f_K((P^{>1})^l)$$ and we have the result by applying the induction hypothesis to $\cl{f_K((P^{<1})^l)}$.
\end{proof}


 \section{Proof of Proposition \ref{prop-examples}}\label{sec ex}

\subsection{Defining new linear functions}\label{sec lin}

 In this subsection, we expose some  preliminary results and questions related to the question:
\begin{question}\label{Q linear} Does $\WR$ defines a total linear function  that is not definable in $\cal R$?
\end{question}

\par \textbf{For the rest of this section, we assume that there is  $\lambda \in \R\setminus \Lambda(\cal R)$.}
\smallskip

\par Let us start a discussion about some restrictions that $\WR$ should satisfy in order to define $x\mapsto \lambda x$. First, a general lemma.
\begin{lemma}\label{Xt bd}
Let $\{g_t:\; t\in A\}$ be a family of functions  definable in \cal R such that for every $t\in A$, $$g_t=\big(x\mapsto \lambda x\big)\upharpoonright (0,x_t).$$
Then there is a uniform bound over the $x_t$'s.
\end{lemma}
\begin{proof}
If there was not, $x\mapsto \lambda x$ would be definable in \cal R.
 \end{proof}
\begin{proposition}\label{fn bd}We assume that  $f:x\mapsto \lambda x$ is definable in $\WR$. Let $\{f_t:(a_t,b_t)\to \R:\;t\in S\}\subseteq \{f_t:\; t\in A\}$ be the families of functions given by DP(I) applied to $f$ (the second one being definable in \cal R). Then there is a uniform bound over $|\dom(f_t)|$ for $t\in S$.
\end{proposition}
\begin{proof}
First of all, we may assume that for every $t\in A$, $f_t$ is linear and that $f_t$ is some restriction of $x\mapsto \lambda x$. We define $g_t:(0,b_t-a_t)$ by $$g_t:x\mapsto f_t(x+a_t)-f_t(a_t).$$
We have the result by applying Lemma \ref{Xt bd}.
\end{proof}
The following proposition shows that more can be said in the d-minimal setting. But first, we need a lemma.
\begin{lemma}\label{unbd} Let $P\subseteq \R$, be a set of dimension $0$ so that $\la\cl{\R},P\ra$ is d-minimal and let $A\subseteq \R$ be a definable unbounded set of dimension $0$. Then the set of the size of complementary intervals of $A$ is unbounded.
\end{lemma}
\begin{proof}
Without loss of generality, we may assume that $A$ is closed. We do an induction on $k$, the number of discrete sets needed to decompose $A$. If $k=1$, then $A$ is closed and discrete and by \cite{Hiero-AEG}, $\WR$ defines the whole projective hierarchy. That is a contradiction with d-minimality.
\par We assume that the result holds for $m<k$. For $Y$ a set of dimension $0$, we denote by $X_Y$ the set of complementary intervals of $Y$ and we denote by $|X_Y|$ the set of size of the complementary intervals of $Y$. Let  $B=\{x\in A:\; \text{$x$ is isolated in $A$}\}$ and $C=A\setminus B$. Note that $C$ is closed and $B$ is discrete. Note also that, using a simple induction, one can easily show that  $C$ admits a decomposition into $k-1$ discrete sets and that $C\subseteq \cl{B}$. Let $$D=C\cup \big( B\setminus \bigcup_{x\in C}[x-1,x+1]\big).$$
It is easy to see that, by construction,  $$ |X_D|+1\leq |X_A|$$ and thus,
if $|X_D|$ is unbounded then $|X_A|$ is unbounded too. Moreover, $C$ does admit a decomposition into $k-1$ discrete sets $C_1,\ldots, C_l$. By definition, for any $i$ $C_i\subseteq \cl{B}$ and thus, $C_i\cup \big(B  \setminus \bigcup_{x\in C}[x-1,x+1]\big)$ is discrete. Thus $D$ does admit a decomposition into $k-1$ discrete sets  and, by applying the induction hypothesis to $D$, we have that $|X_D|$ is unbounded. As said previously, it implies that $|X_A|$ is unbounded and we have the result.
\end{proof}
\begin{proposition}\label{no lin}
Let $P'\subseteq \R$ be a set of dimension $0$. We assume that  $\la\cl{\R},P'\ra$ is d-minimal. Then $x\mapsto \lambda x$ is not definable in $\la\cal R,P'\ra^\#$.
\end{proposition}
\begin{proof}
By Fact \ref{FM}, there is an \cal R-definable function $f:\R^{n+1}\to \R$ so that for every $x\in \R$ $$\lambda x\in \cl{f(x,P'^n)}.$$
By Fact \ref{str function}, there is an \cal R-definable affine function $g:\R^{n+1}\to\R$ and a box $B$ (of radius $\delta$) so that $ f(x)\in g(x)+B$. We may assume that $g$ is affine of the form $y\mapsto \sum_i\mu_iy_i+c$ for some $\mu_i\in I$ and $c\in \R$. We first observe that for every $x\in \R$ $$\lambda x\in g(x,P'^n)+B.$$
 Let $h:\R^n\to\R$, $y\mapsto \sum_{i>1}\mu_iy_i+c$. By Lemma \ref{unbd}, since $h(P'^n)$ is definable in $\la\cl{\R},P'\ra$, there is a complementary interval $E$ of  $\cl{h(P'^n)}$ so that $$|E|>2\delta.$$ Let $y$ be the middle point of $E$. Let $x=y/\lambda$. Since $\lambda\notin I$, $\lambda\neq \mu_1$ and let  $z=\frac{\lambda x+\delta}{\lambda-\mu_1}$. We have that $$\lambda z\in f(z,P'^n)+B=\mu_1 z+h(P'^n)+B$$
 and $$(\lambda-\mu_1)z=y+\delta\in h(P'^n)+B.$$
 This is a contradiction   with the choice of $y$.
\end{proof}
We can then give some sufficient  conditions for $\WR^\#$ to define new linear functions.
\begin{proposition}\label{restr d-min}
Let $P'\subseteq \R$ be an unbounded  set of dimension $0$ so that  $\la\cal R,(x\mapsto \lambda x),P'\ra$ has the interior or isolated point property and such that $\{d(x,P'\setminus \{x\}):\; x\in P'\}$ is bounded. Then $x\mapsto \lambda x$ is definable in $\la\cal R,(x\mapsto \lambda x)_{\res (0,1)},P'\cup \lambda P'\ra^\#$.
\end{proposition}

\begin{proof}First of all, we may assume that $P'$ is closed.
Let $\{(a_t,b_t)\}$ be the (bounded) family of intervals of $(0,\infty)\sm P'$. We first observe that the function $g:x\mapsto \lambda x\upharpoonright P'$ is definable in $\la \cal R,P'\cup \lambda P'\ra^\#$. We define $x\mapsto \lambda x$ on $(0, \infty)$ by $$x\in (a_t,b_t)\mapsto g(a_t)+\lambda (x-a_t).$$
\end{proof}
We can then relate Question \ref{Q linear} to:
\begin{question}
Is there a set $P\subseteq \R$ such that $P$ is closed  unbounded, of dimension $0$, $\la\cal R,(x\mapsto \lambda x),P\ra$ has the interior or isolated point property and  $\{d(x,P\setminus \{x\}):\; x\in P\}$ is bounded?
\end{question}
\subsection{Proof of Proposition \ref{prop-examples}}

\begin{proof}By Theorem \ref{main2}, a smooth function definable with an open  domain definable in  \Cal R needs to be definable in $\la\cal R,\Lambda(\WR),P\ra$.
For ($1$), since for $\lambda\in \R\sm \Q$, $\la\Cal R,\Z,\lambda \Z\ra$ defines a dense-codense set (take for exemple $\Z-\lambda\Z$), $\Lambda(\WR)=\Q$.
\par For ($2$), if $P$ is fast for some expansion of $\Cal R$ then it is fast for \Cal R and we have the result. If $P$ is an iteration sequence for a semibounded expansion of \Cal R, by Remark \ref{rmk alpha^N}, we may assume   it to be of  the form $\alpha ^\N$ and $\la\cal R,\alpha^\N\ra$  as well as $\la\Cal R,\alpha^\Z\ra$ are d-minimal by \cite{VDD1}. If $P$ is an iteration sequence for $\Cal R' $ some non semibounded expansion of $\cal R$, it is sufficient to show the following claim.
\begin{claim} $P$ is fast for $\cal R$.
\end{claim} 
\begin{proof}First, we show that $\Cal R'$ is not linearly bounded. Let $f$ be a non linear function definable in $\Cal R'$. We assume that $f$ is strictly growing, positive and linearly bounded. let $a$ be the infimum of $\{b\in \R: \; bx>f(x) \text{ eventually}\}$. If $b=0$ then $f^{-1}$ is not linearly bounded. If not then let $ g=f-ax$ and $ g^{-1}$ is not linearly bounded.
\par Let $f$ be a non linearly bounded function definable in $\cal R'$. Since $P$ is an iteration sequence for $\Cal R'$, $P$ is of the form $(g^n(a))$ for some $a\in \R$. Moreover, since there is $n\in \N$ so that $g^n>f$ eventually, $g$ is neither linearly bounded and  for every $b\in \R$ $$bx/g(x)\to 0$$
and $P$ is fast for \cal R.   
\end{proof}

 Thus $\la\Cal R,P\ra$ is d-minimal by \cite{FM2}. 
  We apply Proposition \ref{no lin} to get that $\Lambda(\cal R)=\Lambda(\WR)$ and we get the result.
\end{proof}

\subsection{Proof of Proposition \ref{prop-ex-lin}}
\begin{proposition} Let $\cal R$ be a linear reduct of $\cal R_{sb}$ and let $\WR$ be one of ($1$)-($4$) of Example \ref{ex}. Then every definable $\Cal C^1$-function with a domain definable in \Cal R is linear and definable in \Cal R.
\end{proposition}
\begin{proof}
First all these structures are d-minimal by Proposition \ref{ex d-min} and thus satisfy the interior or isolated point property. By Theorem \ref{main2}, a $\cal C^1$ function with a domain definable in \Cal R is linear. 
For ($1$)-($3$), since these structures are reduct of $\la\WR,\cdot_{\res (0,1)^2}\ra$, we have the result by applying Proposition \ref{prop-examples}. For ($4$),  since $P$ is bounded, for every function $f:\R^k\to \R$ definable in \Cal R, we have that $f(P^k)$ is bounded. Let us assume that $g:x\mapsto \lambda x$ is definable in $\WR$. Let $\{X_t:\; t\in S\}$ be the small family of cells given by Fact \ref{DPI} so that $\Gamma(g)=\bigcup_{t\in S}X_t$. Each $X_t$ is the graph of a linear function and let $Y$ be the set of left and right endpoints of $\pi(X_t)$ for every $t\in S$ (and for $\pi$ the projection on the first coordinates). Since $Y$ has dimension $0$, it is bounded and thus one the $X_t$ is cofinal in $\R$. Therefore $g$ is definable in $\cal R$ and that is a contradiction. 
\end{proof}
\section{Discussion and questions}\label{discussion}
In this section we first exhibit a counterexample that shows that Proposition  \ref{R1Z dmin} does not need to hold with the assumption of d-minimality in place of local o-minimality. It also provides some new examples of  d-minimal structures that are not  reducts of a d-minimal expansion of the real field. Second, we  answer a question raised in \cite{ES2}[Question 2.6] about some equivalence between usual notion of semiboundedness in our setting. We exhibit a counterexample that shows that Definition \ref{def sb} is not equivalent to not defining a pole. We finish by raising some questions. 

\begin{proposition}\label{1/N}
The structure $\la\R,<,+,(x\mapsto \lambda x)_{\lambda\in \R},\{1/n:\; n\in \N\}\ra$ is d-minimal and  $\la\R,<,+,\cdot_{\upharpoonright [0,1]^2},1/\N\ra$ is not d-minimal. The latter even defines the trace on bounded sets of any set in the projective hierarchy.
\end{proposition}
FOr the first part we just apply Theorem \ref{T1}.
The second part  follows directly from: 
\begin{theorem}
Let $P\subseteq [0,1] $ be the range  of a decreasing sequence with limit $0$. Then for $\la\R_{sb},P\ra$ to satisfy one of the following properties \begin{enumerate}
\item is d-minimal
\item  is noiseless
\item has the interior or isolated point property
\item  satisfies that every definable set $X\subseteq \R$ has interior or is null (in the sens of Lebeque)
\item defines any bounded set in the projective hierarchy 
\end{enumerate} 
is equivalent for  $\la\cl{\R},P\ra$ to satisfy the corresponding  property among ($1$)-($5$).
\end{theorem}
\begin{proof} Since $\la\R_{sb},P\ra$ is a reduct of $\la\cl{\R},P\ra$, we only have to prove the direct direction and we assume that $\la\R_{sb},P\ra$ has one of  ($1$)-($5$). First of all, we observe that $P$ is $\cl{\R}$-sparse since it is countable.
\par For ($2$)-($4$), let $X\subseteq \R$ be definable in $\la\cl{\R},P\ra$ so that $X$ has dimension $0$ and is ($2$) dense-codense somewhere, ($3$) has  no isolated point, ($4$) is not null. Since these properties are local, we may assume that $X$ is bounded. Moreover, by sparsness, there is $f:\R^k\to \R$ so that $X\subseteq \cl{f(P^k)}$. Since $P^k$ and $X$ are bounded, we may assume that $f$ is semibounded and this is a contradiction with assumptions ($2$)-($4$). 
\par For ($1$), using exactly the same argument we obtain that for every definable family $\{X_t:\, t\in A\}$ of sets of dimension $0$ all contained in a bounded set, there is a uniform bound on the number of discrete sets needed to decompose each $X_t$. In the general case, let $\{X_t:\,t\in A\}$ be a definable family of sets of dimension $0$. We just have to observe  that for any bounded interval $0\in B$, $B\cap (X_t-X_t)$ needs at least the same cardinal of discrete sets (possibly infinitely many) than $X_t$ to be decomposed. This gives us the result by applying the first part to $\{(X_t-X_t)\cap B:\, t\in A\}$ for some bounded interval $0\in B$.
\par For ($5$), by \cite{HW}[Theorem ?], it is sufficient to prove that if $\la \cl{\R},P\ra$ defines a dense $\omega$-orderable set then so does $\la\R_{sb},P\ra$. Let $X$ be dense $\omega$-orderable   definable in $\la\cl{\R},P\ra$. As previously, we may assume that $X$ is bounded and there is a semibounded function $f:[0,1]^k\to \R$ so that $X\subseteq \cl{f(P^k)}$. Since $X$ is dense $\omega$-orderable, so does $f(P^k)$ and we get the result. 
\end{proof}

\medskip
\par In \cite{ES2}, we asked wether or not Definition \ref{def sb} was equivalent to defining a pole. The answer is no and here is an example:
\begin{proposition}
The structure $\la \R,<,+,\{f_t:x\in \R\mapsto tx:\; t\in 2^\Z\}\ra$ defines a pole but does not define a field on $\R$.
\end{proposition} 
\begin{proof}
First, observe that the function $x\in 2^\Z\mapsto x^{-1}$ is definable by $$x\mapsto y \text{ so that $f_x(y)=1$}.$$
Therefore the family of functions $\{g_t:x\in [t,2t)\mapsto -(x-t)/2t+1/2t:\; t\in 2^{-\N}\}$ is also definable and it is not hard to see that $$g:(0,1)\to \R,\, x\in [t,2t)\mapsto g_t(x) \text { is a pole}.$$
Moreover, as shown by Delon in \cite{Delon}, $\la\R,<,+,\{f_t\},2^\Z,(x\mapsto \lambda x)_{\lambda \in \Q}\ra$ is model complete and by a simple analysis of the formulas, we get that  there is no definable non almost everywhere locally linear functions. Since $\la\R,<,+,\{f_t\}\ra$ is a reduct of $\la\cl{\R},2^\Z\ra$, that is d-minimal,  and if a global field is definable then the graph of its multiplication should contain some chunk $Y$ that is semialgebraic and not semilinear. That is a contradiction and we get the result.
\end{proof}
\bigskip\par We finish with some natural questions that arise from the current work.
Around Proposition  \ref{1/N}, we may ask different questions:
\begin{question+}\label{Q1}
Is there a non linear bounded function  $f:(0,1)\to \R$ so that $\la\R,<,+,f\ra$ is o-minimal (and thus, semibounded) and that $\la\R,<,+,f,1/\N\ra$ is d-minimal?
\end{question+}
Actually, Walsberg answered a similar question in \cite{W-sb}[Part $3$] and we can derive an example for Question \ref{Q1}.
\begin{example}
We know that  $\WR=\la \R_{sb}, 2^{-\N}\ra$ is d-minimal. Let $\la \cal R, 1/\N\ra$ be the pushforward of $\WR$ by $\log_2$. This structure answers Question \ref{Q1}.
\end{example}
\begin{question+}
We may ask exactly the same question as \ref{Q1} but with $\R_{vec}$ in place of $\la\R,<,+\ra$.
\end{question+}
\begin{question+}
Is it true that for every  decreasing sequence $(x_n)\to 0$ there is a restricted analytic functions $f$, or just an o-minimal  function $f$ so that  $\la\R_{vec},f,(x_n)\ra$ is d-minimal?
\end{question+}
\begin{rmk}
Of course, we could replace d-minimality by any of the properties described in \cite{Mil1}. That are noiseless, the interior or isolated point property, the interior or lebesgue-null property, defining the projective hierarchy $\ldots$
\end{rmk}
Actually these questions are related to the more general question:
\begin{question+}
Let \cal R be any real closed field expanding $\la\R,<\ra$. Let $P\subseteq \R$ be a sequence set. At which conditions $\la\Cal R,P\ra$ is tame ?
\end{question+}
\medskip
\par We finish by asking a question around a possible generalization of Section \ref{sec lin} to non archimedian settings. 
\begin{question+}\label{Q non arch}
Let $\Cal R=\la R,<,+, \cdot\ra$ be a non archimedian semibounded structure  and $P\subseteq R$ a set of dimension $0$ so that $\WR$ has (DP),  is definably complete and $\Cal R$ defines $(x\mapsto \lambda x)_{(0,\alpha)}$ but does not define $(x\mapsto \lambda x)_{(0,\beta)}$ for every $\beta \gg \alpha$. Is it possible to  define  $(x\mapsto \lambda x)_{(0,\beta)}$ for some $\beta\gg \alpha$ but not to define $(x\mapsto \lambda x)$? 
\end{question+}
\begin{rmk}
Assuming that there is a d-minimal  expansion $\WR=\la\cal R,P\ra^\#$ of $\R_{gp}$ that defines $x\mapsto \lambda x$ (and such that $\lambda \notin \Lambda(\Cal R)$), we may easily define a structure that answers Question \ref{Q non arch}.  By (DPI) in $\WR$, there is $\{X_t:\, t\in S\}$ a definable family of restrictions of $x\mapsto \lambda x$ to some domain. Let $Y$ be the set of endpoints of
 $\{\pi(X_t):\; t\in S\}$ and $Z=\Gamma(x\mapsto \lambda x)_{\res Y}$. 
 Let $Z_{\omega}=\big\{\{\alpha\}\times (Z\cap (0,\alpha)\times (0,\lambda\alpha)):\; \alpha \in \R\big\}$. Let $\cal M$ be a non standard model of $\WR^\#$ and let $\cal R'=\la R,<,+,\ldots\ra$ be the non standard model of $\Cal R$ contained in it. Let $Z_{\omega}^*$ be the interpretation of $Z_{\omega}$ in \cal M and let $\alpha \gg 1$ be a short element of $R$. It is then easy to see that $\la \Cal R', Z^*_{\omega,\alpha}\ra$ has all the desired properties (as a reduct of \cal M), defines $(x\mapsto \lambda x)_{\res (0,\alpha)}$  but does not define the total $x\mapsto \lambda x$. 
\end{rmk}

\end{document}